\documentclass[11pt]{article}

\usepackage{amsmath, amsfonts, amsthm, graphicx,amssymb}
\usepackage[top=21mm, bottom=22mm, left=21mm, right=21mm]{geometry}
\usepackage{comment}
\usepackage{hyperref}
\hypersetup{
	colorlinks=true,
	linkcolor=blue,
	filecolor=magenta,      
	urlcolor=cyan,
	citecolor=blue
}
\usepackage{enumerate}
\usepackage{aliascnt} 
\usepackage[none]{hyphenat}
\usepackage{tikz, tikz-3dplot}
\usetikzlibrary{calc, patterns,intersections, scopes}
\usetikzlibrary{positioning,arrows,shapes,decorations.markings,decorations.pathreplacing, decorations.pathmorphing,matrix,patterns}
\tikzstyle{vtx}=[circle,draw=black,fill=black,inner sep=0,minimum size=5pt,text=white,font=\footnotesize]

\newtheorem{theorem}{Theorem}[section]

\newaliascnt{lemma}{theorem}
\newtheorem{lemma}[lemma]{Lemma}
\aliascntresetthe{lemma}
\newaliascnt{claim}{theorem}
\newtheorem{claim}[claim]{Claim}
\aliascntresetthe{claim}
\newaliascnt{question}{theorem}

\aliascntresetthe{question}

\usepackage{cleveref}
\crefname{lemma}{lemma}{lemmas}
\Crefname{lemma}{Lemma}{Lemmas}
\crefname{claim}{claim}{claims}
\Crefname{claim}{Claim}{Claims}
\crefname{question}{question}{questions}
\Crefname{question}{Question}{Questions}
\newcommand{\refD}[1]{Definition~\ref{#1}}

\theoremstyle{definition}
\newtheorem{definition}{Definition}

\newcommand{\cV}{V}
\newcommand{\cS}{N}
\newcommand{\cB}{I}

\newcommand{\eps}{\varepsilon}

\newcommand{\cP}{\mathcal{P}}

\DeclareMathOperator{\cng}{cong}
\DeclareMathOperator{\cube}{cube}
\DeclareMathOperator{\dist}{dist}

\DeclareMathOperator{\conv}{conv}
\DeclareMathOperator{\supp}{supp}

\title{Random 0/1-polytopes expand rapidly}
\author{He Guo\thanks{Ume\r{a} University, \emph{e-mail}: \texttt{he.guo@umu.se}, Research supported in part by the Kempe Foundation JCSMK23-0055.}\and 
Istv\'an Tomon\thanks{Ume\r{a} University, \emph{e-mail}: \texttt{istvan.tomon@umu.se}, Research supported by the Swedish Research Council VR 2023-03375.}}
\date{}

\begin{document}

\maketitle
\sloppy

\begin{abstract}
A 0/1-polytope is the convex hull of a subset $V\subseteq \{0,1\}^n$. A celebrated conjecture of Mihail and Vazirani asserts that the graph of every 0/1-polytope has edge-expansion at least 1.  In this paper, we show that typical 0/1-polytopes have significantly stronger expansion. Specifically, if $V$ is formed by sampling each vertex of $\{0,1\}^n$ independently with constant probability $p$, then with high probability the edge-expansion is $\Theta(n)$ for $p \in (1/2, 1)$, and $n^{\Theta(\log \log n)}$ for $p \in (0, 1/2)$. This improves the previously best known bound $\Omega(1)$ due to Ferber, Krivelevich, Sales and Samotij.
\end{abstract}

\section{Introduction}

A \emph{0/1-polytope} in $\mathbb{R}^n$ is the convex hull of a subset of the $n$-dimensional hypercube $\{0,1\}^n$.  These polytopes play a central role in polyhedral combinatorics and combinatorial optimization. Their significance stems from the fact that many algorithmic problems require finding an optimal subset of items from a finite universe, which can be naturally formulated as an integer linear program minimizing a linear cost function $c^Tx$ over a feasible set $V \subseteq \{0,1\}^n$. Because minimizing a linear function over a discrete set is mathematically equivalent to minimizing it over its convex hull, understanding the geometric structure of $\conv(V)$ is crucial for designing efficient algorithms. Consequently, much of polyhedral combinatorics is dedicated to identifying the exact facet descriptions of specific 0/1-polytopes, such as the matching, matroid, and independent set polytopes.

The \emph{graph} (or 1-skeleton) of a polytope $P$ is the graph whose vertices are the 0-dimensional faces of $P$, and whose edges are its 1-dimensional faces. The classical Simplex algorithm solves linear programs by traversing the edges of this graph until reaching an optimal vertex. To guarantee that such path-following algorithms can be fast, the graph of the polytope must have a small diameter. For 0/1-polytopes, a celebrated theorem by Naddef \cite{Naddef} establishes that the diameter is bounded by its dimension, and is therefore at most $n$. Geometrically, this guarantees that an optimal solution is never more than $n$ steps away from any starting point.

While the diameter governs worst-case path lengths, the \emph{edge-expansion} (or Cheeger constant) characterizes the global connectivity. For a graph $G$, the edge-expansion is defined as 
$$h(G)=\min \left\{\frac{|\partial(U)|}{|U|}: U\subseteq V(G),\; 1\le |U|\leq \frac{1}{2}|V(G)|\right\},$$ where $\partial(U)$ is the set of edges with exactly one endpoint in $U$. The edge-expansion of 0/1-polytopes plays a key role in the theory of approximate counting and sampling. Large edge-expansion ensures the rapid mixing of Markov chain random walks, a property utilized to generate elements of combinatorial objects uniformly at random. For instance, analyzing random walks on the matching polytope allowed Jerrum, Sinclair, and Vigoda \cite{JS1,JS2} to design a polynomial-time randomized approximation scheme (FPRAS) for computing the permanent of nonnegative matrices. Motivated by these connections, Mihail and Vazirani \cite{FM,Mihail} famously conjectured that $h(G)\geq 1$ for the graph of \emph{every} 0/1-polytope. The classical result of Harper \cite{Harper} shows that $h(G)=1$ if $V=\{0,1\}^n$ is the entire hypercube. The conjecture of Mihail and Vazirani remains a major open problem, though it  has been established for several specific classes: stable set polytopes, perfect matching polytopes \cite{Kaibel}, order ideal polytopes, matching polytopes, independent set polytopes \cite{Mihail0} and balanced matroid base polytopes \cite{FM}. In a recent breakthrough, Anari, Liu, Gharan, and Vinzant \cite{ALOV} proved the conjecture for matroid base polytopes. 

However, the conjecture remained open even for \emph{random 0/1-polytopes}, whose study was proposed by Kaibel and Remshagen \cite{KR} and Gillmann \cite{Gillmann}. Let $V$ be generated by sampling each vertex of $\{0,1\}^n$ independently with probability $p$, and let $G$ be the graph of $\conv(V)$.  Leroux and Rademacher \cite{LR} proved that $h(G)= \Omega(1/n)$ with high probability (w.h.p.). Here and later, an event, or more precisely a sequence of events $\{A_n\}_{n\in \mathbb{N}}$ happens \emph{w.h.p.} if $\lim_{n\rightarrow \infty}\mathbb{P}(A_n)=1$. This was improved by Ferber, Krivelevich, Sales and Samotij \cite{FKSS}, who established that w.h.p. $h(G)=\Omega(1)$.

In this paper, we significantly improve these results. We show that when $p$ is bounded away from 1, the edge-expansion is much larger, and we determine the correct order of magnitude of $h(G)$ across a wide range of probabilities $p=p(n)$. This resolves questions posed in \cite{FKSS} in a strong form, and  it shows that the Mihail-Vazirani conjecture is true for random 0/1-polytopes.

\begin{theorem}\label{thm:main}
For every $\eps>0$ the following holds. Let $V$ be formed by sampling each vertex of $\{0,1\}^n$ independently with probability $p=p(n)$, and let $G$ be the graph of~$\conv(V)$. Then w.h.p.:
\begin{enumerate}[(a)]
    \item \label{eq:maina} $h(G)=\Theta(n)$ if $p\in (1/2+\eps,1-\eps)$,
    \item \label{eq:mainb} $h(G)=n^{\Theta(\log\log n+\log (1/p))}$ for $p\in (n^{-0.05},1/2-\eps)$.
\end{enumerate}
The constant hidden by $\Theta(\cdot)$ may depend on $\eps$.
\end{theorem}

The upper bounds on $h(G)$ follow from the degrees in $G$, as we show in \Cref{lemma:upperbound}. For $p\in (1/2+\eps,1]$, w.h.p. there exist vertices whose hypercube neighbors are all contained in $V$, ensuring a degree of $n$ and thus $h(G) \leq n$. Conversely, for $p \in (0, 1/2)$, $G$ contains few edges connecting vertices of Hamming-distance much greater than $\log\log n + \log(1/p)$, which limits the average degree and implies the corresponding upper bound in (\ref{eq:mainb}). Our main contributions are the matching lower bounds, whose proof we outline at the end of this section. We remark that in the case $p\in [1/2-\eps,1/2+\eps]$, which is not covered by \Cref{thm:main}, our proof  shows that $h(G)=\Omega(n)$, and the only obstacle of larger expansion is the existence of vertices whose hypercube neighbors are almost all sampled.

\Cref{thm:main} only applies when $p>n^{-O(1)}$. However, our results extend to smaller values of $p$ as well via a projection argument introduced in \cite{LR}, and strengthened in \cite{FKSS}.

\begin{theorem}\label{thm:general}
Let $c>0$, then there exists $C>0$ such that the following holds. Let $p=p(n)\in (2^{-0.9n},n^{-c})$, and  let $V$ be formed by sampling each vertex of $\{0,1\}^n$ independently with probability $p$. Then w.h.p.  the graph $G$ of~$\conv(V)$ satisfies
$$h(G)\geq n^{C\log n}.$$
\end{theorem}

For exponentially small $p$, even tighter characterizations exist. In case $p\leq 2^{-5n/6}$, Bondarenko and Brodskii~\cite{BB} proved that $G$ is w.h.p. a clique. Babecki, Elling, and Ferber \cite{BEF} refined this, identifying a threshold $\delta \approx 0.8295$ such that $G$ is a clique if $p \leq 2^{-(\delta+\epsilon)n}$. They also showed that for $p < 2^{-n(1/2+\epsilon)}$, every vertex has degree $(1-o(1))|V|$, implying $h(G) = \Omega(|V|)$. Collectively, these results show that random 0/1-polytopes exhibit super-polynomial edge-expansion for all $p<1/2-\eps$ as long as the size of~$V$ permits.

\subsection{Proof outline}

To establish that a graph $G$ has large edge-expansion, it is sufficient to construct a multicommodity flow with low maximum congestion (see \Cref{sect:flow}). Specifically, for every pair of vertices $x,y\in V(G)$, we  distribute a unit of flow across the paths connecting $x$ and $y$ such that the congestion of every edge $e$, i.e. the total flow traversing $e$, is small. In case $V$ is sampled randomly from $\{0,1\}^n$ with probability $p$, and $G$ is the graph of $\conv(V)$, we construct such a flow as follows. We focus on the regime $p\in (n^{-0.05},1/2-\eps)$, noting that the analysis for $p\in (1/2+\eps,1-\eps)$ is similar. 

Let $d\ll \log\log n+\log (1/p)$  be an odd integer, and let $Q_n^d$ be the distance-$d$ graph on $\{0,1\}^n$, where two vertices are adjacent  if their Hamming-distance is $d$. We first define a flow in $Q_n^d$ where each edge has congestion at most $n2^n/\binom{n}{d}$ (\Cref{lemma:congQnd}), which implies $h(Q_n^d)\geq \binom{n}{d}/(2n)$ (\Cref{thm:congestion}). We then consider a particular  subgraph~$G_d(V)$ of the intersection of~$G$ and~$Q_n^d$ (see \Cref{sect:main}). We argue that the flow of every path in $Q_n^d$ can be locally redistributed to paths in $G_d(V)$ such that no edge becomes too congested. Essentially, we replace each edge $e$ of $Q_n^d$ by a path of length 7 in $G_d(V)$, whose vertices are close in Hamming-distance to the endpoints of~$e$, and stitch these path together to form new paths. This argument draws inspiration from a classic work by H\r{a}stad and Leighton \cite{HL}: even if a constant fraction of nodes in a hypercube-shaped network are randomly destroyed, the remaining nodes can still efficiently simulate the original network by locally rerouting communication through surviving paths.

\subsection{Paper organization}

In the next section, we introduce the main tools and definition used in the proof of \Cref{thm:main}, including concentration inequalites, expansion and network flows, and polytopes. We also prove the upper bounds of \Cref{thm:main}, and we prove \Cref{thm:general} assuming \Cref{thm:main}. Finally, we present the proof of the lower bounds of \Cref{thm:main} in \Cref{sect:main}.

\section{Preliminaries}

\subsection{Concentration inequalities}

We use standard concentration inequalities, which we collect here for the reader's convenience.

\begin{lemma}[Multiplicative Chernoff bound]\label{lemma:chernoff}
Let $X_1,\dots,X_n$ be independent indicator random variables, let $X=X_1+\dots+X_n$ and $\mu=\mathbb{E}X$. If $0\leq \delta$, then
$$\mathbb{P}(X>(1+\delta)\mu)\leq \exp(-\delta^2 \mu /(2+\delta)).$$
In particular, if $2\mu<t$, then
$$\mathbb{P}(X>t)\leq \exp(-t/3).$$
Moreover, if $0\leq \delta<1$, then
$$\mathbb{P}(X<(1-\delta)\mu)\leq \exp(-\delta^2 \mu/2).$$
In particular,
$$\mathbb{P}(X<\mu/2)\leq \exp(-\mu/8).$$
\end{lemma}

\begin{lemma}[Chernoff bound]\label{lemma:chernoff_concrete}
Let $X\sim \mbox{Binom}(n,p)$. Then for any $c<p$, 
$$\mathbb{P}(X<cn)\leq \exp(-nD(c||p)),$$
where $D(c||p)=c \log\left(\frac{c}{p}\right) + (1-c) \log\left(\frac{1-c}{1-p}\right)$. In particular, if $p>0.5$, then  there exists $\alpha>0$ such that
$$\mathbb{P}(X<\alpha n)\leq 2^{-(1+\alpha)n}.$$
\end{lemma}

\begin{lemma}[McDiarmid's inequality]\label{lemma:mcdiarmid}
    Let $f:\mathcal{X}_1\times\dots\times \mathcal{X}_n\rightarrow\mathbb{R}$. Assume that if $x,x'\in \mathcal{X}_1\times\dots\times\mathcal{X}_n$ only differ in the $i$-th coordinate, then $|f(x)-f(x')|\leq \Delta_i$. If $X_1\in \mathcal{X}_1,\dots,X_n\in\mathcal{X}_n$ are independent random variables and $X=f(X_1,\dots,X_n)$, then
    $$\mathbb{P}(X\leq \mathbb{E}(X)-t)\leq \exp\left(-\frac{2t^2}{\sum_{i=1}^n \Delta_i^2}\right).$$
\end{lemma}

Finally, we present a result about the concentration of sums of random variables with few dependencies. Similar results in a more general setting are discussed by Janson \cite{janson}.

\begin{lemma}\label{lemma:grid}
Let $m\leq n$, and let $X_{a,b}$, $(a,b)\in [m]\times [n]$, be Bernoulli random variables with probability~$p$ such that~$X_{a,b}$ is mutually independent from $X_{a',b'}$ if $a\neq a'$ and $b\neq b'$. Let $X=\sum_{(a,b)\in [m]\times [n]}X_{a,b}$, and let $\mu=\mathbb{E}(X)$. Then
$$\mathbb{P}(X\leq \mu/2)\leq n\exp(-\mu/(8n)).$$
\end{lemma}

\begin{proof}
    We can partition $[m]\times [n]$ into $n$ sets $I_1,\dots,I_n$, each of size $m$, such that no two members of $I_i$ have the same coordinate for every $i\in [n]$. Then
    $$X_i=\sum_{(a,b)\in I_i}X_{a,b}$$
    is the sum of independent indicator random variables with expectation $pm$. Hence, by the multiplicative Chernoff bound,
    $$\mathbb{P}(X_i\leq pm/2)\leq \exp(-pm/8).$$
    But then
    $$\mathbb{P}(X\leq \mu/2)\leq \mathbb{P}(\cup_{i=1}^n\{X_i\le pm/2\}) \le \sum_{i=1}^n\mathbb{P}(X_i\leq pm/2)\leq n\exp(-pm/8)=n\exp(-\mu/(8n)),$$
    which completes the proof.
\end{proof}

\subsection{Edge expansion and multicommodity flows}\label{sect:flow}

Let $G$ be a graph. Given a subset of vertices $U\subseteq V(G)$, we denote by $U^c$ the complement of $U$, that is, $U^c=V(G)\setminus U$. The \emph{edge-boundary} of $U$ is the set of edges with exactly one endpoint in $U$, formally,
$$\partial_G(U):=\partial(U)=\{xy\in E(G): x\in U, y\in U^c\}.$$
The \emph{edge-expansion} of $G$, also known as the \emph{Cheeger constant} of $G$, is 
$$h(G):=\min\left\{ \frac{|\partial(U)|}{|U|}: U\subseteq V(G),\; 1\le |U|\leq \frac{1}{2}|V(G)|\right\}.$$
The edge-expansion of a graph can be lower bounded with the help of network flows. The fundamentals of this theory are established by Leighton and Rao \cite{LRao}. Nevertheless, we present the necessary notions and results in a self-contained manner. 

For $x,y\in V(G)$, let $\cP_{x,y}$ denote the set of paths in $G$ with endpoints $x$ and $y$, and let $\cP=\bigcup_{x,y\in V(G)}\cP_{x,y}$. A function $\varphi:\cP\rightarrow \mathbb{R}_{\geq 0}$ is an \emph{all-pair unit-demand multicommodity flow}, which we abbreviate as \emph{A-flow}, if for every pair of distinct $x,y\in V(G)$,
$$\sum_{P\in \cP_{x,y}}\varphi(P)=1.$$
%for every $x,y\in V(G)$, $x\neq y$. 
The \emph{congestion} of an edge $e\in E(G)$ with respect to $\varphi$ is
$$\cng_{\varphi}(e)=\sum_{P\in \cP: e\in P}\varphi(P).$$
Moreover, the \emph{congestion} of $\varphi$ is the maximum congestion of an edge, that is, 
$$\cng(\varphi)=\max_{e\in E(G)}\cng_{\varphi}(e).$$
Finally, the \emph{congestion} of $G$ is the minimum congestion over every A-flow $\varphi$ on $G$, that is,
$$\cng(G)=\min_{\varphi} \cng(\varphi).$$

\begin{theorem}\label{thm:congestion}
Let $\varphi$ be an A-flow on an $n$-vertex graph $G$. Then
$$h(G)\geq \frac{n}{2\cng(\varphi)}.$$
\end{theorem}

\begin{proof}
 Let $U\subseteq V(G)$ be such that $|U|\leq n/2$. We calculate the total weight of paths, whose endpoints are separated by $U$. First of all, we have
\begin{equation}\label{equ:cong1}
    \sum_{(x,y)\in U\times U^c}\sum_{P\in \cP_{x,y}}\varphi(P)=|U|(n-|U|).
\end{equation}
On the other hand, note that every path $P\in \cP_{x,y}$, where $(x,y)\in U\times U^c$, contains at least one edge of $\partial(U)$. Therefore,
\begin{equation}\label{equ:cong2}
\sum_{(x,y)\in U\times U^c}\sum_{P\in \cP_{x,y}}\varphi(P)\leq \sum_{e\in \partial(U)}\cng_{\varphi}(e)\leq |\partial(U)|\cng(\varphi).
\end{equation}
Comparing (\ref{equ:cong1}) and (\ref{equ:cong2}), we get
$$\frac{|\partial(U)|}{|U|}\geq \frac{n-|U|}{\cng(\varphi)}\geq \frac{n}{2\cng(\varphi)}.$$
As this holds for every $U$,  we have $h(G)\geq n/(2\cng(\varphi))$.
\end{proof}

 We prove via a simple concentration argument that we can efficiently simulate an A-flow by picking a single path between any two vertices.

\begin{lemma}\label{lemma:path-system}
Let~$G$ be a graph on~$n$ vertices and let~$\varphi$ be an A-flow such that $\cng(\varphi)\leq C$ for some~$C>6\log n$. Then for every $x,y\in G$ one can select $P_{x,y}\in \cP_{x,y}$ such that every edge of $G$ is contained in at most~$2C$ of the paths $\{P_{x,y}\}_{x,y\in V(G)}$.
\end{lemma}

\begin{proof}
For every $x,y\in V(G)$, $\varphi$ is a probability distribution on $\cP_{x,y}$. Pick $P_{x,y}$ randomly from $\cP_{x,y}$ with respect to $\varphi$, independently from other pairs $(x',y')$. For an edge $e\in E(G)$, let $X_e$ denote the number of paths $P_{x,y}$ containing $e$, then
$$\mathbb{E}(X_e)=\cng_{\varphi}(e)\leq C.$$
As $X_e$ is the sum of independent indicator random variables, we can apply the multiplicative Chernoff bound (\Cref{lemma:chernoff}) to get
$$\mathbb{P}(X_e>2C)\leq e^{-C/3}<\frac{1}{n^2}.$$
Hence, by the union bound, we have 
$$\mathbb{P}(\forall e\in E(G): X_e<2C)>0,$$
which means that there is a choice of the paths $P_{x,y}$ such that every edge is contained in at most $2C$ of these paths.
\end{proof}

\subsection{The hypercube and 0/1-polytopes}

The \emph{Hamming-distance} of $x,y\in \{0,1\}^n$, denoted by $\dist_H(x,y)$, is the number of indices $i\in [n]$ such that $x_i\neq y_i$. The \emph{Hamming-sphere} and \emph{Hamming-ball} of radius $d$ are  
$$S_d(x)=\Big\{y\in \{0,1\}^n:\dist_H(x,y)=d \Big\}\mbox{\ \ \ \ and \ \ \ }B_d(x)= \Big\{y\in \{0,1\}^n:\dist_H(x,y)\leq d \Big\},$$
and we write $S_d:=S_d(0)$ and $B_d:=B_d(0)$ for simplicity. Define $x\oplus y\in \{0,1\}^n$ as $(x\oplus y)_i=x_i+y_i \pmod{2}$ for each $i\in [n]$. The \emph{support} of~$x$ is the set of indices $i\in [n]$ with $x_i=1$, denoted by $\supp(x)$.

The $n$-dimensional hypercube graph $Q_n$ is the graph on vertex set $\{0,1\}^n$, where two vertices are joined by an edge if their Hamming-distance is 1. Given $x,y\in \{0,1\}^n$, let~
$$\cube(x,y)=\Big\{z\in \{0,1\}^n: \forall i\in [n], \min(x_i,y_i)\leq z_i\leq \max(x_i,y_i)\Big\}.$$
Then $\cube(x,y)$ is the smallest subcube of $\{0,1\}^n$ that contains both $x$ and $y$, and it is isomorphic to a $\dist_H(x,y)$-dimensional hypercube. 

Let $V \subseteq \{0,1\}^n$ and $P = \conv(V)$. A subset $F \subseteq P$ is a face of $P$ if it is the intersection of $P$ with a supporting hyperplane; equivalently, there exists a linear inequality $a^T x \le b$ that is valid for all $x \in P$ and $F = \{x \in P : a^T x = b\}$. The vertices of $F$ are precisely $V \cap F$. The dimension of a face $F$, denoted by $\dim(F)$, is defined as the dimension of its affine hull. Faces of dimension 0 are the vertices of $P$, while faces of dimension 1 are its edges. The graph (or 1-skeleton) of $P$ is the graph $G = (V, E)$ where~$V$ is the set of vertices and~$E$ is the set of edges of~$P$. In this context, two vertices $u, v \in V$ form an edge in~$G$ if and only if the line segment $[u, v]$ is a 1-dimensional face of~$P$. This is equivalent to the existence of a linear inequality $a^T x \le b$ such that $a^T u = a^T v = b$ and $a^T w < b$ for all $w \in V \setminus \{u, v\}$.

The following simple observations provide sufficient conditions for determining whether two vertices in a 0/1-polytope are adjacent or non-adjacent. While these represent special subcases of conditions established in \cite{FKSS}, they are sufficient for the proofs presented in this paper.

\begin{lemma}\label{lemma:edge}
Let $V\subseteq \{0,1\}^n$, and let $u,v\in V$ such that $\cube(u,v)\cap V=\{u,v\}$. Then $uv$ is an edge of the graph of $\conv(V)$.
\end{lemma}

\begin{proof}
By the symmetry of the hypercube, we may assume that $u=0$. Let $a\in\mathbb{R}^n$ be the vector defined as $a_i=1-v_i$ for $i\in [n]$. Then $a^Tu=a^Tv=0$. However,  if $w\not\in \cube(u,v)$, then there exists some $i$ such that $w_i=1$ and $v_i=0$, which means that $a^Tw>0$. Therefore, as $\cube(u,v)\cap V=\{u,v\}$, we have~$a^Tw>0$ for every $V\setminus \{u,v\}$. Thus, $\{u,v\}$ is an edge of the graph of $\conv(V)$.
\end{proof}

\begin{lemma}\label{lemma:non-edge}
Let $V\subseteq \{0,1\}^n$, $u,v\in V$, and let $G$ be the graph of $\conv(V)$. 
\begin{enumerate}[(a)]
    \item\label{eq:lemma:non-edgea} If $S_1(u)\subseteq V\setminus\{v\}$, then $uv\not\in E(G)$.

    \item\label{eq:lemma:non-edgeb} If there exist $s,t\in \cube(u,v)\setminus \{u,v\}$ such that $s\oplus t=u\oplus v$,  then $uv\not\in E(G)$.
\end{enumerate}
\end{lemma}

\begin{proof}
By the symmetry of the hypercube, we may assume that $u=0$ without loss of generality. Assume that $uv$ is an edge of $G$, then there exists a vector $a$ such that $a^Tv=0$ and $a^Tz>0$ for every $z\in V\setminus \{u,v\}$.

\eqref{eq:lemma:non-edgea} We can write $s_1+\dots+s_d=v$ with suitable $s_1,\dots,s_d\in S_1$. But if $s_1,\dots,s_d\in V\setminus \{v\}$, then
$$0<a^Ts_1+\dots+a^Ts_d=a^Tv=0,$$
a contradiction.

\eqref{eq:lemma:non-edgeb}  By the assumption $u=0$, for every $s,t\in \cube(0,v)$,  $s\oplus t=v$ implies $s+t=v$.  But this is impossible as otherwise $$0=\langle a,v\rangle=\langle a,s\rangle+\langle a,t\rangle>0,$$
a contradiction.
\end{proof}

Next, we prove the upper bounds of \Cref{thm:main}, which are simple corollaries of \Cref{lemma:non-edge}.

\begin{lemma}\label{lemma:upperbound}
Let $V$ be formed by sampling each vertex of $\{0,1\}^n$ independently with probability $p=p(n)$, and let $G$ be the graph of $\conv(V)$. Then w.h.p.:
\begin{enumerate}[(a)]
    \item \label{eq:lemma:upperbounda} $G$ contains a vertex of degree $n$ if $p\in (1/2+\eps,1]$;
    \item \label{eq:lemma:upperboundb} The average degree of $G$ is at most $n^{O(\log \log n+\log (1/p))}$.
\end{enumerate}
\end{lemma}

\begin{proof}

\eqref{eq:lemma:upperbounda}  If $B_1(u)\subseteq V$ for some vertex~$u$, then \Cref{lemma:edge} and part~\eqref{eq:lemma:non-edgea} of \Cref{lemma:non-edge} ensures that~$S_1(u)$ is the entire neighborhood of~$u$ in~$G$. It is a standard exercise to show that if $p\geq 1/2+\eps$, then there exists a vertex $u$ satisfying $B_1(u)\subseteq V$ with high probability. Indeed, let $I_u$ be the indicator random variable of the event $B_1(u)\subseteq V$, and let $X=\sum_{u\in \{0,1\}^n}I_u$. Then $\mathbb{E}I_u=p^{n+1}$, and thus $\mathbb{E}X=2^np^{n+1}$. Consider the correlation of the variables $I_u$ and $I_v$. If $u$ and $v$ have Hamming-distance~1 or~2, then $\mbox{Cov}(I_u,I_v)\leq \mathbb{E}I_uI_v=p^{2n}$, otherwise $I_u$ and $I_v$ are independent. Therefore,
$$\mbox{Var}(X)=\sum_{u,v\in \{0,1\}^n}\mbox{Cov}(I_u,I_v)=\sum_{u\in \{0,1\}^n}\mbox{Var}(I_u)+\sum_{\substack{u,v\in \{0,1\}^n:\\ \dist_H(u,v)\in \{1,2\}}}\mbox{Cov}(I_u,I_v)\leq 2^n p^{n+1}+2^nn^2 p^{2n},$$
where we use that $\mbox{Var}(I_u)\leq \mathbb{E}I_u=p^{n+1}$, and there are at most $2^nn^2$ pairs $(u,v)$ of Hamming-distance~1 or~2. By Chebyshev's inequality,
$$\mathbb{P}(X=0)\leq \mathbb{P}(|X-\mathbb{E}X|\geq \mathbb{E}X)\leq \frac{\mbox{Var}(X)}{(\mathbb{E}X)^2}\leq \frac{2^np^{n+1}+2^nn^2p^{2n}}{2^{2n}p^{2n+2}}=\frac{1}{2^np^{n+1}}+\frac{n^2}{2^np^2}.$$
For $p\geq 1/2+\eps$, both terms converge to 0, so w.h.p. $X>0$. 

\eqref{eq:lemma:upperboundb} Let $d$ be a positive integer, and let $u,v\in \{0,1\}^n$ such that $\dist_H(u,v)=d$. Then $\cube(s,t)$ is a $d$-dimensional hypercube. The relation $s\oplus t=u\oplus v$ partitions $\cube(u,v)\setminus \{u,v\}$ into $2^{d-1}-1$ pairs~$(s,t)$. By part~\eqref{eq:lemma:non-edgeb} of \Cref{lemma:non-edge}, if any of these pairs is contained in~$V$, then~$uv$ is not an edge of~$G$. Therefore,
$$\mathbb{P}(uv\in E(G))\leq (1-p^2)^{2^{d-1}-1}\leq \exp(-p^2(2^{d-1}-1)).$$
In case $d\geq 2\log_2\log n+2\log_2 (1/p)=:L$, we have $p^2(2^{d-1}-1)\geq 2d\log n$, so 
$$\mathbb{P}(uv\in E(G))\leq \exp(-2d\log n)=n^{-2d}.$$
Hence, the expected number of edges $uv$ of $G$ such that $\dist_H(u,v)\geq L$ can be bounded as
$$2^n\sum_{d\geq L}\binom{n}{d}n^{-2d}=o(2^n).$$
Hence by Markov's inequality, w.h.p. $G$ has $o(2^n)$ such edges. 
On the other hand, the number of edges~$uv$ of~$G$ with $\dist_H(u,v)\le L$ is at most
%But then w.h.p., the number of edges of $G$ is at most 
$$2^{n}\sum_{d\leq L}\binom{n}{d}\leq 2^n\cdot n^{L}\leq 2^n\cdot n^{2\log_2\log n+2\log_2 (1/p)}.$$
The above arguments imply that w.h.p. the average degree of~$G$ is~$n^{O(\log \log n+\log (1/p))}$.
\end{proof}

Finally, we prove that \Cref{thm:main} indeed implies \Cref{thm:general}. In order to show this, we use the following projection lemma of \cite{FKSS}. Given $x\in \{0,1\}^n$, let $\pi_d(x)\in \{0,1\}^{n-d}$ denote the projection of $x$ onto the first $n-d$ coordinates.

\begin{lemma}[\cite{FKSS}]\label{lemma:projection}
Let $V\subseteq \{0,1\}^n$, and let $G$ be the graph of $\conv(V)$. Let $W=\pi_d(V)$ and let $H$ be the graph of $\conv(W)$. Then for every $A\subseteq V$,
$$|\partial_{G}(A)|\geq |\partial_{H}(\pi_d(A))|.$$
\end{lemma}

\begin{proof}[Proof of \Cref{thm:general} assuming \Cref{thm:main}]
Let $c=0.05$, and let $\delta=100n^{-c}$. Let $d=\lfloor \log_2 (\frac{\delta}{p})\rfloor$, and write $\pi=\pi_d$ and $L=2^d$. Then $d\leq 0.91n$ by the assumption $p>2^{-0.9n}$, and $\frac{\delta}{2}\leq p L\leq \delta$. Observe that every $y\in \{0,1\}^{n-d}$ is included in $\pi(V)$ independently with probability 
$$\hat{p}=1-(1-p)^{L}\in (pL-O(\delta^2),pL),$$
so $W\sim V_{n-d,\hat{p}}$. Let $H$ be the graph of the convex hull of $W=\pi(V)$. As 
$$\hat{p}\in (40n^{-c},100n^{-c})\subseteq ((n-d)^{-c},1/3),$$ we can apply \Cref{thm:main}, which implies that w.h.p.  $h(H)=(n-d)^{\Omega(\log (1/\hat{p}))}=n^{\Omega(\log n)}$.

As $\mathbb{E}|V|=p2^n=pL2^{n-d}$, the multiplicative Chernoff bound (\Cref{lemma:chernoff}) implies that w.h.p. 
$$|V|\leq 1.01pL2^{n-d}.$$ 
Similarly, as $\mathbb{E}|W|=\hat{p}2^{n-d}$, the multiplicative Chernoff bound implies that w.h.p. 
$$|W|\geq 0.995 \hat{p} 2^{n-d}\geq 0.99pL2^{n-d}.$$ 
Furthermore, let $X$ be the number of $x\in \{0,1\}^{n-d}$ such that $|\pi^{-1}(x)|\geq n$. Then 
$$\mathbb{E}X\leq 2^{n-d}\cdot p^{n}\binom{L}{n}\leq 2^{n-d} (pL)^{n}\leq 2^{-n},$$
which implies that w.h.p. $X=0$. We show that the conditions $h(H)=n^{\Omega(\log n)}$, $|W|\geq 0.9|V|$, and $X=0$ imply that $h(G)=n^{\Omega(\log n)}$. Then we are done as these conditions are all satisfied w.h.p..

Let $A\subseteq V$ such that $|A|\leq |V|/2$, and let $B=\pi(A)$. Then \Cref{lemma:projection} ensures that
$$|\partial_G(A)|\geq |\partial_{H}(B)|.$$
On the one hand, we have $|B|\leq |A|\leq \frac{1}{2}|V|\leq \frac{2}{3}|W|$. Therefore, if $|B|\geq |W|/2$, then $$|\partial_G(A)|\geq |\partial_{H}(B)|=|\partial_{H}(B^c)|\geq |B^c|h(H)\geq \frac{|W|}{3}h(H)\geq \frac{|A|}{3}h(H).$$
This gives $|\partial_{G}(A)|/|A|\geq \frac{1}{3}h(H)$. On the other hand, if $|B|\leq |W|/2$, then  the assumption $X=0$ implies $|B|\geq |A|/n$. Therefore,
$$|\partial_G(A)|\geq |\partial_{H}(B)|\geq |B|h(H)\geq \frac{|A|}{n}h(H).$$
Hence, $|\partial_G(A)|/|A|\geq \frac{1}{n}h(H)$. Thus, 
$$h(G)\geq \frac{1}{n}h(H)=n^{\Omega(\log n)},$$
which completes the proof.
\end{proof}

\section{Edge-expansion of random polytopes}\label{sect:main}
In this section, we prove the lower bounds of \Cref{thm:main}. We start by introducing some further notation, which is used extensively throughout this section.     Let $d,n$ be positive integers and $V\subseteq \{0,1\}^n$.
\begin{itemize}
\item $Q_n^d$ is the Hamming-distance $d$ graph. Formally, the vertex set of $Q_n^d$ is $\{0,1\}^n$, and~$x$ and~$y$ are adjacent if $\dist_H(x,y)=d$.
\vspace{-5pt}
\item $G_d(V)$ is  the graph with vertex set $V$ in which $x,y\in V$ are adjacent if 
$$\dist_H(x,y)=d\mbox{ and }\cube(x,y)\cap V=\{x,y\}.$$
\vspace{-20pt}
\item $N_{d,V}(x)$ is the set of neighbours of $x\in V$ in $G_{d}(V)$.
\end{itemize} 

By \Cref{lemma:edge}, $G_d(V)$ is a subgraph of the graph~$G$ of $\conv(V)$. Therefore, in order to show that~$G$ has large edge-expansion, it suffices  to bound the expansion of $G_d(V)$ for some $d$.  Following the framework of \Cref{thm:congestion}, we bound $h(G_d(V))$ by constructing an A-flow in $G_d(V)$ with small maximum congestion. Observe that $G_d(V)$ is a subgraph of $Q_n^d$ as well. The latter is highly symmetric, which is a property we can exploit to construct low-congestion A-flows. We show that an ideal A-flow in $Q_n^d$ can be rerouted into a valid A-flow in $G_d(V)$ via local modifications, ensuring that the congestion remains within a constant factor of the original. Note that since $Q_n^d$ is disconnected for even $d$, we restrict our analysis to odd $d$ in certain cases.

In the rest of this section, we also use the following shorthand:
\[N=N_n=2^n \quad \text{ and }\quad D=D_{n,d}=\binom{n}{d}.\] 
Note that~$Q_n^d$ is an $N$-vertex $D$-regular graph. First, we  show that~$Q_n^d$ has an A-flow with congestion at most $nN/D$.

\begin{lemma}\label{lemma:congQnd}
If $d$ is odd and $d<n$, then $\cng(Q_n^d)\leq \frac{nN}{D}.$
\end{lemma}

\begin{proof}
For $x,y\in \{0,1\}^n$, let $\mathcal{R}_{x,y}$ be the set of shortest paths between $x$ and $y$ in $Q_n^d$. Observe that every member of $\mathcal{R}_{x,y}$ has length at most $n$. Let $\mathcal{R}=\bigcup_{x,y\in \{0,1\}^n} \mathcal{R}_{x,y}$, and define the A-flow $\varphi:\mathcal{R}\rightarrow [0,1]$ such that $\varphi(P)=|\mathcal{R}_{x,y}|^{-1}$ for every $P\in \mathcal{R}_{x,y}$. By the symmetry of the hypercube, the congestion of every edge of $Q_n^d$ is the same value $C=\cng(\varphi)$. Therefore,
$$\frac{ND}{2}C=e(Q_n^d)C=\sum_{e\in E(Q_n^d)}\cng_{\varphi}(e)=\sum_{P\in\mathcal{R}}|P|\varphi(P)\leq n\sum_{P\in \mathcal{R}}\varphi(P)=n\binom{N}{2}.$$
Comparing the left- and right-hand side, we get $C\leq \frac{nN}{D}$.
\end{proof}

Let $p=p(n)\in (0,1)$, and let~$\cV=\cV_{n,p}$ be the random subset of $\{0,1\}^n$, where each vertex is sampled independently with probability~$p$. We always assume that $n$ is sufficiently large. Write
\[ q=q_{d}(p)=p(1-p)^{2^d-2}.\]
Then given an edge $xy\in Q_n^d$, i.e.,  $\dist_H(x,y)=d$, we have 
$$\mathbb{P}\Big(xy\in E(G_d(\cV))\mid\ x\in \cV\Big)=q.$$

Let $x,y\in \cV$, then our goal is to define a weighting of the paths between~$x$ and~$y$ in $G_d(\cV)$. Given a~$Q_n^d$-path~$P$ between~$x$ and~$y$, we reroute this path using  edges of~$G_d(\cV)$. We do this by replacing every vertex~$x_i$ of the path~$P$ with a vertex $z_i\in \cV$ whose Hamming-distance to~$x_i$ is~$d$, and then we connect~$z_i$ and~$z_{i+1}$ with a special path of length~7 in~$G_d(\cV)$. This motivates the following definition. 

\begin{definition}[Pure paths]\label{def:purepath}
Let $x,y\in \{0,1\}^n$ such that $\dist_H(x,y)=3d$. Let $t,u,f_1,f_2,f_3\in S_d$ satisfy that
$x\oplus y=f_1\oplus f_2\oplus f_3$, where $\supp(f_1)$ is the set of the first~$d$ elements of~$\supp(x\oplus y)$, $\supp(f_2)$ is the set of the second~$d$ elements of~$\supp(x\oplus y)$, $\supp(f_3)$ is the set of the last~$d$ elements of~$\supp(x\oplus y)$, and~$\supp(t)$,~$\supp(u)$,~$\supp(x\oplus y)$ are pairwise disjoint.  Then the path with vertices
\begin{align*}
x_0&=x\\
x_1&=x\oplus t\\
x_2&=x\oplus t\oplus u\\
x_3&=x\oplus t\oplus u\oplus f_1\\
x_4&=x\oplus t\oplus u\oplus f_1\oplus f_2\\
x_5&=y\oplus t\oplus u=x\oplus t\oplus u\oplus f_1\oplus f_2\oplus f_3\\
x_6&=y\oplus u\\
x_7&=y
\end{align*}
is a \emph{pure path} between~$x$ and~$y$.
\end{definition}

It is easy to check that a pure path is indeed a path of length 7 in $Q_n^d$, and $f_1,f_2,f_3$ is uniquely determined for every pair~$(x,y)$ with~$\dist_H(x,y)=3d$. We present a simple lemma about the number of pure paths containing a given edge.

\begin{lemma}\label{lemma:path_count}
Let $e\in E(Q_n^d)$. Then $e$ is contained in at most $O(3^{3d} D^2)$ pure paths.
\end{lemma}

\begin{proof}
Extend the notion of pure paths in~\refD{def:purepath} to allow any choice of $f_1,f_2,f_3\in S_d$ satisfying $f_1\oplus f_2\oplus f_3=x\oplus y$, and call such paths \emph{semi-pure}. Let~$C$ be the number of semi-pure paths containing~$e$, then by the symmetry of $Q_n^d$,~$C$ does not depend on the choice of~$e$. The total number of semi-pure paths in~$Q_n^d$ is at most 
$$ND^3\cdot D^2\cdot \frac{(3d)!}{(d!)^3}\leq 3^{3d}ND^5,$$
as there are at most $N\binom{n}{3d}<ND^3$ choices for the pair $(x,y)$, $D^2$ choices for the pair $(t,u)$, and exactly~$\binom{3d}{d,d,d}=\frac{(3d)!}{(d!)^3}$ choices for $f_1,f_2,f_3$. By counting the total number of edges of all semi-pure paths in two ways, we arrive to the inequality 
$$C\cdot \frac{ND}{2}\leq 7\cdot 3^{3d}ND^5,$$
which gives $C=O(3^{3d}D^4)$.
\end{proof}

The following lemma is our main technical result showing that if $x,y\in \cV$ such that $\dist_H(x,y)=3d$ and at least a constant proportion of the neighbours of~$x$ and~$y$ in~$Q_n$ are not sampled in~$\cV$, then~$x$ and~$y$ can be connected by many pure paths in~$G_d(\cV)$ with very high probability. To this end, given~$\alpha\in [0,1]$, say that 
$$x\in \cV\mbox{ is }\alpha\mbox{-\textit{full} if }|S_1(x)\cap \cV|\geq (1-\alpha)n.$$

\begin{lemma}\label{lemma:main}
Let $\alpha\in (0,1]$ be fixed, then there exists $c>0$ such that the following holds. Let  $d \ge 3$ and $p\in (0,1)$ such that $q=q_d(p)\geq n^{-0.1}$ and $(\alpha/8)^d\geq n^{-0.1}$. Let $x,y\in \cV=\cV_{n,p}$ such that $\dist_H(x,y)=3d$. Let $Y$ be the number of pure paths in $G_d(\cV)$ between~$x$ and~$y$. Then
$$\mathbb{P}\left(Y\geq c^{d}q^7D^2 \Big|\ x,y\mbox{ are not }\alpha\mbox{-full}\right)\geq 1-\exp(-n^{1.5}).$$
\end{lemma}

\begin{proof}
    Fix some $\cS_x\subseteq S_1(x)$ and $\cS_y\subseteq S_1(y)$ with $\min(|\cS_x|,|\cS_y|)\geq \alpha n$, and condition on the event 
    $$\mathcal{E}=\Big\{\cS_x=S_1(x)\setminus \cV\mbox{ and }\cS_y=S_1(y)\setminus \cV\Big\}$$ 
    in the rest of the proof. In other words, the elements of $S_1(x)\cup S_1(y)$ are revealed and $\cS_x\cup\cS_y$ are the set of the elements that are not included in~$\cV$.  Writing $e_1,\dots,e_n$ for the standard unit vectors,  every element of $S_1(x)$ is of the form $x\oplus e_i$ for some $i\in [n]$. 
    Let $\cB'_x=\{i\in [n]:e_i\oplus x\in \cS_x\}$ and $\cB'_y=\{i\in [n]:e_i\oplus y\in \cS_y\}$. 
    Since $|\cB_x'|=|\cS_x|\ge \alpha n$, $|\cB_y'|=|\cS_y|\ge \alpha n$, and $\supp(x\oplus y)=3d < \alpha n/4$, we can
    fix some $\cB_x\subseteq \cB'_x\setminus \supp(x\oplus y), \cB_y\subseteq \cB'_y\setminus \supp(x\oplus y)$ such that $|\cB_x|=|\cB_y|=\alpha n/4$ and $\cB_x\cap \cB_y=\emptyset$. 
    
    Define 
    $$T:=\{t\in S_d: \supp(t)\subseteq \cB_x\}\quad\mbox{ and }\quad U:=\{u\in S_d: \supp(u)\subseteq \cB_y\}.$$
    Here, $T$ and $U$ are the set of potential vectors $t$ and $u$ which we use to build a pure path connecting $x$ and $y$. We have 
    $$|T|=|U|=\binom{\alpha n/4}{d}\geq 2c_0^d D$$ for $c_0=\alpha/8$. Let~
    $$T^*=\{t\in T: x\oplus t\in N_{d,V}(x)\}\quad  \mbox{ and }\quad U^*=\{u\in U:  y\oplus u\in N_{d,V}(y)\}.$$
    \begin{claim}
    The following holds:
    $$\mathbb{P}\Big(\min\big(|T^*|,\;|U^*|\big)\leq c_0^dqD\Big)\leq \exp(-n^{1.55}).$$ 
    \end{claim}
    \begin{proof}
    Recall that~$ab$ is an edge of~$G_d(\cV)$ if $\dist_H(a,b)=d$ and $\cube(a,b)\cap \cV=\{a,b\}$. Therefore, given~$t\in T$, we have
    $$\mathbb{P}(t\in T^*)=p\cdot(1-p)^{2^d-d-2}\geq q.$$
    This is true as by conditioning on $\mathcal{E}$, we already ensured that the $d$ elements of $\cube(x,x\oplus t)\cap S_1(x)$ are not in~$\cV$. Therefore,  
    $$\mathbb{E}(|T^*|)\geq 2c_0^dqD.$$
     Observe that~$|T^*|$ is a function of the events $\{z\in \cV\}$, where $2\leq \dist_H(x,z)\leq d$. Let $z\in S_r(x)$ with $2\leq r\leq d$. Then $z$ is contained in $\cube(x,x\oplus t)$ for at most $\binom{n-r}{d-r}\leq D(\frac{d}{n})^r=:\Delta_z$ values of $t\in T$. It means that the random variable $|T^*|$ is a function of the independent events $\{z\in \cV\}$, with the  property that changing the outcome of $\{z\in \cV\}$ can change the value of $|T^*|$ by at most $\Delta_z$. By McDiarmid's inequality (\Cref{lemma:mcdiarmid}), 
    $$\mathbb{P}\left(|T^*|\leq \frac{1}{2}\mathbb{E}|T^*|\right)\leq \exp\left(-\frac{(\mathbb{E}|T^*|)^2}{2\sum_{r=2}^d\sum_{z\in S_r(x)}\Delta_z^2}\right).$$
    Here, 
    $$\sum_{r=2}^d\sum_{z\in S_r(x)}\Delta_z^2\leq \sum_{r=2}^d \binom{n}{r}D^2\left(\frac{d}{n}\right)^{2r}<D^2\sum_{r=2}^d\left(\frac{d^2}{n}\right)^r\leq \frac{2D^2d^4}{n^2}.$$
    Plugging this back to the inequality above, we have
     $$\mathbb{P}\left(|T^*|\leq \frac{1}{2}\mathbb{E}|T^*|\right)\leq \exp\left(-\frac{\mathbb{E}(|T^*|)^2}{2D^2d^4/n^2}\right)\leq \exp\left(-\frac{2c_0^{2d}q^2n^2}{d^4}\right)\leq \frac{1}{2}\exp(-n^{1.55}),$$
     where the last inequality follows by our assumption that $c_0^d=(\alpha/8)^d\geq n^{-0.1}$ and $q\geq n^{-0.1}$. Similarly,
     $$\mathbb{P}\left(|U^*|\leq c_0^dqD\right)\leq \frac{1}{2}\exp(-n^{1.55}),$$
     finishing the proof.
     \end{proof}
    In what follows, we fix some 
    \[T_0\subseteq T\quad \text{ and } \quad U_0\subseteq U\]
    such that $\min(|T_0|,|U_0|)\geq c_0^{d}qD$. In the rest of the proof, we further condition on the event 
    $$\mathcal{E}'=\{T^*=T_0\quad \mbox{ and }\quad U^*=U_0\}.$$ In other words, we revealed the events $\{z\in \cV\}$ for some subset of elements of $B_d(x)\cup B_d(y)$. Let 
    $$Z:=\left\{(t,u)\subseteq T_0\times U_0: x\oplus t\oplus u\in N_{d,V}(x\oplus t)\mbox{ and }y\oplus u\oplus t\in  N_{d,V}(y\oplus u)\right\}.$$

    \begin{claim}The following holds:
    $$\mathbb{P}\left(|Z|\leq \frac{1}{2}c_0^{2d}q^4D^2\right)\leq \exp(-n^{1.6}).$$ 
    \end{claim}

    \begin{proof}
    Let $(t,u)\in T_0\times U_0$, and let $X_{t,u}$ be the indicator random variable of the event $\{(t,u)\in Z\}$. Write $$C_{t,u}=\cube(x\oplus t,x\oplus t\oplus u)\setminus \{x\oplus t\}\mbox{\ \ \ \ and \ \ \ }D_{t,u}=\cube(y\oplus u,y\oplus u\oplus t)\setminus \{y\oplus u\}.$$ Then every element of $C_{t,u}$ and $D_{t,u}$ has Hamming-distance more than $d$ to both $x$ and $y$, and also $C_{t,u}\cap D_{t,u}=\emptyset$ (because $\supp(t),\supp(u)$, and $\supp(x\oplus y)$ are pairwise disjoint). Therefore, no element of $C_{t,u}$ and $D_{t,u}$ has been revealed, showing that
    $$\mathbb{E}X_{t,u}=q^2.$$
     Thus,
     $$\mathbb{E}|Z|=q^2|T_0||U_0|\geq c_0^{2d}q^4D^2.$$ 
     
     Recall that $\supp(t)\cap \supp(u)=\emptyset$ for every pair $(t,u)\in T_0\times U_0$, as $\supp(t)\subseteq \cB_x$ and $\supp(u)\subseteq \cB_y$. This implies that if $(t,u),(t',u')\subseteq T_0\times U_0$, then $C_{t,u}\cap C_{t',u'}=\emptyset$ if $t\neq t'$, and $D_{t,u}\cap D_{t',u'}=\emptyset$ if $u\neq u'$. Moreover, $C_{t,u}\cap D_{t',u'}=\emptyset$ if $(t,u)\neq (t',u')$, because every vertex of $C_{t,u}$ has Hamming-distance at most~$2d$ from~$x$, while it has Hamming-distance more than $4d$ from $y$, and similarly for $D_{t',u'}$. Therefore, $X_{t,u}$ is mutually independent from $\{X_{t',u'}:t\neq t'\text{ and } u\neq u'\}$. By \Cref{lemma:grid}, using that $\max(|T_0|,|U_0|)\leq D$, we have
     $$\mathbb{P}\left(|Z|\leq \frac{1}{2}\mathbb{E}|Z|\right)\leq D\exp\left(-\frac{\mathbb{E}|Z|}{8D}\right)<D\exp(c_0^{2d}q^4 D)\leq \exp(-n^{1.6}).$$
     The last inequality holds by our assumption that $d\geq 3$, which ensures $D=\Omega(n^3)$, and $q,c_0^d\geq n^{-0.1}$.
    \end{proof}

Fix some 
\[Z_0\subseteq Z\]
such that $|Z_0|\geq \frac{1}{2}c_0^{2d}q^4D^2$. In the rest of the proof, we further condition on the event~$$\mathcal{E}''=\{Z=Z_0\}.$$ So far, we revealed only events~$\{z\in \cV\}$ for some subset of $B_{2d}(x)\cup B_{2d}(y)$.

    Let $f_1,f_2,f_3\in S_d$ be as described in the definition of a pure path, that is, $x\oplus y=f_1\oplus f_2\oplus f_3$, $\supp(f_1)$ are the first~$d$ elements of $\supp(x\oplus y)$, $\supp(f_2)$ are the second~$d$ elements of $\supp(x\oplus y)$, and $\supp(f_3)$ are the last~$d$ element of $\supp(x\oplus y)$.  Let 
    \[W\subseteq Z_0\]
    be the set of pairs $(t,u)\in Z_0$ such that 
    $$x\oplus t\oplus u,\ x\oplus t\oplus u\oplus f_1,\ x\oplus t\oplus u\oplus f_1\oplus f_2,\ y\oplus u\oplus t$$ is a path of length 3 in~$G_d(\cV)$. To simplify notation, write $x_1,x_2,x_3,x_4$ for these four vertices, and let 
    $$L_{t,u}=\cube(x_1,x_2)\cup\cube(x_2,x_3)\cup \cube(x_3,x_4)\setminus \{x_1,x_4\}.$$
    Then $(t,u)\in W$ only depends on the events $\{z\in \cV\}$ for $z\in L_{t,u}$. Moreover, every element of $L_{t,u}$ has Hamming-distance at least $2d+1$ to both $x$ and $y$, using that~$\supp(t)$ and~$\supp(u)$ are disjoint from $\supp(x\oplus y)=\supp(f_1)\cup\supp(f_2)\cup\supp(f_3)$. Therefore,
    $$\mathbb{P}((t,u)\in W)=p^2(1-p)^{3\cdot 2^d-6}=\frac{1}{q}q^3,$$
    and thus
    $$\mathbb{E}|W|\geq \frac{1}{p}q^3|Z_0|\geq \frac{1}{2p}c_0^{2d}q^7D^2.$$ 
    \begin{claim}
     If $(t,u),(t',u')\in Z_0$ are distinct, then  $L_{t,u}$ and $L_{t',u'}$ are disjoint.
    \end{claim}
    \begin{proof}
       Let $I=[n]\setminus \supp(x\oplus y)$, and for $z\in \{0,1\}^n$, write $z_I$ for the restriction of $z$ to the coordinates in~$I$. Using that $\supp(f_1),\supp(f_2),\supp(f_3)$ are disjoint from $I$, we have  $z_I=(x\oplus t\oplus u)_I$ for every~$z\in L_{t,u}$. But observe that $(t\oplus u)_I\neq (t'\oplus u')_I$, which shows that $z_I\neq z'_I$ if $z\in L_{t,u}$ and $z'\in L_{t',u'}$. Indeed, assume that $(t\oplus u)_I=(t'\oplus u')_I$, then $(t\oplus t')_I=(u\oplus u')_I$. Here, $\supp(t),\supp(t')\subseteq I_x\subseteq I$ and $\supp(u),\supp(u')\subseteq I_y\subseteq I$ with $I_x\cap I_y=\emptyset$. This implies that $t\oplus t'=u\oplus u'$ also holds. But $\supp(t\oplus t')\subseteq I_x$ and $\supp(u\oplus u')\subseteq I_y$. Hence, we must have $t\oplus t'=0$ and $u\oplus u'=0$, which implies $(t,u)=(t',u')$.
    \end{proof}
    The previous claim implies that the events $\{(t,u)\in W\}$ for~$(t,u)\in Z_0$, are independent. Thus~$|W|$ is the sum of independent indicator random variables, so by the multiplicative Chernoff bound (\Cref{lemma:chernoff}),
    $$\mathbb{P}\left(|W|\leq \frac{\mathbb{E}|W|}{2}\right)\leq \exp\left(-\frac{1}{8}\mathbb{E}|W|\right)\leq \exp\left(-\frac{1}{16p}c_0^{2d}q^7 D^2\right)\leq \exp(-n^{1.6}).$$
    Each element of $W$ gives rise to a pure path between $x$ and $y$. Thus,  there are at least $\frac{1}{4p}c_0^{2d}q^7D^2>\frac{1}{4}c_0^{2d}q^7D^2$ pure paths between $x$ and $y$ in $G_d(\cV)$ with probability at least $1-\exp(-n^{1.6})$, conditioned on the events $\mathcal{E},\mathcal{E'},\mathcal{E''}$. But then we also have 
    $$\mathbb{P}\left(Y\geq \frac{1}{4}c_0^{2d}q^7D^2 \Big|\ x,y\mbox{ are not }\alpha\mbox{-full}\right)\geq (1-\exp(-n^{1.55}))(1-\exp(-n^{1.6}))^2\geq 1-\exp(-n^{1.5}).$$ Thus, the choice $c=c_0^2/2$ satisfies the required properties, completing the proof of~\Cref{lemma:main}. 
\end{proof}

Next, we collect a number of typical properties of $V$, which will ensure that we can build an A-flow with small congestion. As $\alpha$-full elements form the bottleneck of the expansion, we need to take special care of them. First, we consider the range $p\in (\eps,1-\eps)$ in~\Cref{lemma:conditions1}, then the range $p\in (0,1/2-\eps)$ in~\Cref{lemma:conditions2}.

\begin{lemma}\label{lemma:conditions1}
Let $\eps>0$, then there exists $\alpha,c,K>0$ such that the following hold. Let $p\in (\eps,1-\eps)$, then w.h.p
\begin{enumerate}[(i)]
    \item\label{a:1} $|V|\geq cN$;
    \item\label{a:2} for every $x\in \{0,1\}^n$, $|S_3(x)\cap V|\geq cn^3$;
    \item\label{a:3} for every $x\in \{0,1\}^n$, the number of $\alpha$-full elements in $B_3(x)$ is at most $K$;
    \item\label{a:4} for every $x\in V$, if $x$ is $\alpha$-full, then there are at least $cn^3$ vertices in $S_3(x)$ that are the endpoints of some path of length $3$ in $G_1(V)$ starting at $x$;  
    \item\label{a:5} for every $x,y\in V$ such that $\dist_H(x,y)=9$ and $x,y$ are not $\alpha$-full, there are at least $cn^6$ pure paths in $G_3(V)$ between $x$ and $y$;
     \item\label{a:6} for every $x\in V$, if $x$ is not $\alpha$-full, then $|N_{3,V}(x)|\geq cn^3$.
\end{enumerate}
\end{lemma}

\begin{proof}
Let $\alpha=\eps/3$ and $K=\lceil 100/\eps\rceil$. Then we show that any sufficiently small $c$ with respect to $\eps$ suffices.

\eqref{a:1} As $|V|$ is the sum of indicator random variables and $\mathbb{E}|V|=pN>\eps N$, the multiplicative Chernoff bound implies $\mathbb{P}(|V|\leq \eps N/2)\leq \exp(-\eps N/8)=o(1)$. Therefore, any $c<\eps/2$ suffices.

\eqref{a:2} As $|S_3(x)\cap V|$ is the sum of independent indicator random variables and $\mathbb{E}|S_3(x)\cap V|=p\binom{n}{3}>\eps n^3/10$, the multiplicative Chernoff bound implies that $\mathbb{P}(|S_3(x)\cap V|\leq \eps n^3/20)\leq \exp(-\eps n^3/80)\leq \exp(-n^2)$. Therefore, by the union bound, $$\mathbb{P}(\exists x, |S_3(x)\cap V|\leq  \eps n^3/20)\leq 2^n\exp(-n^2)=o(1).$$ Therefore, any $c<\eps/20$ suffices.

\eqref{a:3} Let $x_1,\dots,x_K\in \{0,1\}^n$ be distinct vertices, and let 
$$T_i:=S_1(x_i)\setminus \bigcup_{j:j\neq i}S_1(x_j).$$
As $|S_1(x_i)\cap S_1(x_j)|\leq 2$ for every $1\leq i<j\leq n$, we have $|T_i|\geq n-2K$. Moreover,
$$\mathbb{P}(x_i\mbox{ is }\alpha\mbox{-full})\leq \mathbb{P}(|T_i\setminus V|\leq \alpha n).$$
As $|T_i\setminus V|$ is the sum of indicator random variables with mean $(1-p)|T_i|\geq (1-p)n-2K\geq \eps n-2K>2\alpha n$, we can apply the multiplicative Chernoff  bound (\Cref{lemma:chernoff}):
$$\mathbb{P}(|T_i\setminus V|\leq \alpha n)\leq \exp(-((1-p)n-2K)/8)\leq \exp(-(1-p)n/16).$$
As $T_1,\dots,T_K$ are disjoint, the events $\{|T_i\setminus V|\leq \alpha n\}$ are independent, so we get
$$\mathbb{P}(x_1,\dots,x_K\mbox{ are }\alpha\mbox{-full})\leq \exp(-(1-p)Kn/16)<4^{-n},$$
where the last inequality follows by our choice $K\geq100/\eps$. 

Say that a $K$-tuple $(x_1,\dots,x_K)$ is \emph{concentrated} if there exists a vertex $x$ such that $x_1,\dots,x_K\in B_3(x)$, and say that $(x_1,\dots,x_K)$ is bad if $x_1,\dots,x_K$ are all $\alpha$-full. The number of concentrated $K$-tuples is at most $2^nn^{3K}$, so the expected number of concentrated bad $K$-tuples is at most $(2/4)^nn^{3K}=o(1)$. Thus, by Markov's inequality, w.h.p., there are no concentrated bad $K$-tuples, finishing the proof.

\eqref{a:4} Let $x\in \{0,1\}^n$ be $\alpha$-full for some $\alpha$. For $i=1,2,3$, let $m_i=\alpha n (\eps n/12)^{i-1}$ (the definition of~$m_i$ will become clear later). Let $T_1^x=T_1:=S_1(x)\cap V$, and if $T_{i-1}$ is already defined, then let~$U_i$ be the set of vertices $x\in S_i(x)$ that have Hamming-distance~1 to some element of $T_{i-1}$. As each element of $S_{i-1}(x)$ has $n-i+1\geq n/2$ neighbours in $S_i(x)$ in the graph $Q_n$, and every element of $S_i(x)$ has $i\leq 3$ neighbours in $S_{i-1}(x)$, we have $|U_i|\geq |T_{i-1}|n/6$. Define $T_i=T_i^x=U_i\cap V$,  then $\mathbb{E}(T_i|T_{i-1},\dots,T_1)\geq |T_{i-1}|pn/6\geq |T_{i-1}|\eps n/6$. Therefore,as $m_{i}=m_{i-1}\eps n/12$,  the multiplicative Chernoff bound (\Cref{lemma:chernoff}) gives
$$\mathbb{P}\Big(|T_i|\leq m_i \mid |T_{i-1}|\geq  m_{i-1},\dots,|T_1|\geq m_1\Big)\leq \exp(m_i/6)\leq \exp(-n^{1.6}).$$
Thus,
\begin{align*}
\mathbb{P}\Big(|T_3|\geq m_3 \mid  |T_1|\geq m_1\Big)&\geq \mathbb{P}\Big(|T_3|\geq m_3 \mid |T_{2}|\geq m_{2},|T_1|\geq m_1\Big)\cdot \mathbb{P}\Big(|T_2|\geq m_2 \mid |T_1|\geq m_1\Big) \\
&\geq (1-\exp(-n^{1.6}))^2\geq 1-\exp(-n^{1.5}).
\end{align*}
There is a path of length $3$ in $G_1(V)$ between $x$ and any element of $T_3$. Choose $c$ such that $c\leq (\alpha \eps/20)^3$, then we have $m_3\geq (\alpha \eps/12)^{3}n^{3}\geq c n^3$. Hence, the probability that there exists $x\in V$ such that $x$ is $\alpha$-full  (equivalently $|T_1^x|\geq m_1=\alpha n$) and  $|T_3^x|\leq cn^3$ is  at most $2^n\exp(-n^{1.5})=o(1)$ by the union bound. This implies that \eqref{a:4} happens w.h.p..

\eqref{a:5}  For a pair of vertices $x,y\in V$ with $\dist_H(x,y)=9$, let $Y_{x,y}$ denote the number of pure paths in  $G_3(V)$ with endpoints $x$ and $y$. By \Cref{lemma:main}, noting that the conditions on $p$ and $d=3$ are trivially satisfied, there exists $c_0=c_0(\alpha)>0$ such that 
$$\mathbb{P}\left(Y_{x,y}\leq c_0^{d}q^7D^2 \mid x,y\mbox{ are not }\alpha\mbox{-full}\right)\leq \exp(-n^{1.5}).$$ Here, we have $c_0^{d}q^7D^2\geq cn^6$ for some $c$ sufficiently small with respect to $\eps$. Say that a pair $(x,y)$ is bad if $Y_{x,y}\leq cn^6$ and $x,y$ are not $\alpha$-full. Then
\begin{align*}
\mathbb{P}((x,y)\mbox{ is bad})=&\mathbb{P}(Y_{x,y}\leq cn^6\mbox{ and }x,y\mbox{ are not }\alpha\mbox{-full})\\
\leq &\mathbb{P}(Y_{x,y}\leq cn^6 \mid x,y\mbox{ are not }\alpha\mbox{-full})\leq \exp(-n^{1.5}).
\end{align*}
Hence, the expected number of bad pairs is at most $2^n\binom{n}{9}\exp(-n^{1.5})=o(1)$, so w.h.p. there are no bad pairs.

\eqref{a:6} Assume that \eqref{a:2}, \eqref{a:3} and \eqref{a:5} hold. Then given $x\in V$ that is not $\alpha$-full, there exists $y\in V$ such that $\dist_H(x,y)=9$ and $y$ is also not $\alpha$-full. But then there are at least $cn^6$ pure paths in $G_d(V)$ between $x$ and $y$. As each edge from $x$ appears in at most $n^3$ pure paths (recalling that in~\refD{def:purepath}, for each~$t$, there are at most~$D<n^3$ choices of~$u$), this implies that $|N_{d,V}(x)|\geq cn^3$.
\end{proof}

\begin{lemma}\label{lemma:conditions2}
 Let $\eps>0$, then there exists $\alpha,c>0$ such that the following holds. Let  $d \ge 3$ and $p\in (0,1/2-\eps)$ such that $q=q_d(p)\geq n^{-0.1}$ and $(\alpha/8)^d\geq n^{-0.1}$. Then w.h.p.:
 \begin{enumerate}[(i)]
    \item\label{b:1} $|V|\geq pN/2$;
    \item\label{b:2} for every $x\in \{0,1\}^n$, $|S_d(x)\cap V|\geq pD/2$;
    \item\label{b:3} there are no $\alpha$-full vertices.
    \item\label{b:4} for every $x,y\in V$ such that $\dist_H(x,y)=3d$, there are at least $c^{d}q^7D^2$ pure paths in $G_d(V)$ between $x$ and $y$;
    \item\label{b:5} for every $x\in V$, $|N_{d,V}(x)|\geq c^{d}q^7D$.
\end{enumerate}
 \end{lemma}

 \begin{proof}
 We highlight that the condition $q\geq n^{-0.1}$ also ensures that $p\geq n^{-0.1}$.
 
     \eqref{b:1} As $|V|$ is the sum of indicator random variables and $\mathbb{E}|V|=pN$, the multiplicative Chernoff bound implies $\mathbb{P}(|V|\leq pN/2)\leq \exp(-pN/8)=o(1)$.

\eqref{b:2} As $|S_d(x)\cap V|$ is the sum of independent indicator random variables and $\mathbb{E}|S_d(x)\cap V|=pD$, the multiplicative Chernoff bound implies that $\mathbb{P}(|S_d(x)\cap V|\leq pD/2)\leq \exp(-pD/8)\leq \exp(-n^2)$. Therefore, by the union bound, $$\mathbb{P}(\exists x, |S_d(x)\cap V|\leq pD/2)\leq 2^n\exp(-n^2)=o(1).$$

\eqref{b:3} For $x\in \{0,1\}^n$, let $Z_x=|S_1(x)\setminus V|$. Then $Z_x\sim \mbox{Binom}(n,1-p)$ and $1-p\geq 1/2+\eps$. Therefore, by \Cref{lemma:chernoff_concrete}, there exists $\alpha>0$, depending only on $\eps$, such that 
$$\mathbb{P}(Z_x\leq \alpha n)\leq 2^{-(1+\alpha)n}.$$
By the union bound,
$$\mathbb{P}(\exists x, Z_x\leq \alpha n)\leq 2^n\cdot 2^{-(1+\alpha)n}=2^{-\alpha n},$$
As $x\in V$ is $\alpha$-full if $Z_x\leq \alpha n$, this implies that w.h.p. there are no $\alpha$-full vertices. 

\eqref{b:4}  For a pair of vertices $x,y\in \{0,1\}^n$ with $\dist_H(x,y)=3d$, let $Y_{x,y}$ denote the number of pure paths in  $G_d(V)$ with endpoints $x$ and $y$. By \Cref{lemma:main}, there exists $c=c(\alpha)>0$ such that 
$$\mathbb{P}\left(Y_{x,y}\leq c^{d}q^7D^2\mid x,y\mbox{ are not }\alpha\mbox{-full}\right)\leq \exp(-n^{1.5}).$$ Say that a pair $(x,y)$ is bad if $Y_{x,y}\leq c^{d}q^7D^2$ and $x,y$ are not $\alpha$-full. Then
\begin{align*}
\mathbb{P}((x,y)\mbox{ is bad})=&\mathbb{P}\left(Y_{x,y}\leq c^{d}q^7D^2\mbox{ and }x,y\mbox{ are not }\alpha\mbox{-full}\right)\\
\leq &\mathbb{P}\left(Y_{x,y}\leq c^{d}q^7D^2 \mid x,y\mbox{ are not }\alpha\mbox{-full}\right)\leq \exp(-n^{1.5}).
\end{align*}
Hence, the expected number of bad pairs is at most $2^n\binom{n}{d}\exp(-n^{1.5})=o(1)$, so w.h.p. there are no bad pairs. But by \eqref{b:3}, there are also no $\alpha$-full vertices, and thus there are the required number of paths between every pair of vertices $x,y\in V$, $\dist_H(x,y)=3d$.

\eqref{b:5} Assume that \eqref{b:4} holds. Then given $x\in V$ that is not $\alpha$-full, there exists $y\in V$ such that $\dist_H(x,y)=3d$ and $y$ is also not $\alpha$-full. But then there are at least $c^{d}q^7D^2$ pure paths in $G_d(V)$ between $x$ and $y$. Since in~\refD{def:purepath}, for each~$t$, there are at most~$D$ choices of~$u$, this implies that $|N_{d,V}(x)|\geq c^{d}q^7pD$.

 \end{proof}

Now we are ready to prove the main result of this section, which almost immediately implies \Cref{thm:main}. Given $x,y,x',y'\in \{0,1\}^n$, say that $xy$ \emph{avoids} $x'y'$ if $\supp(x\oplus y)\cap \supp(x'\oplus y')=\emptyset$.

\begin{theorem}\label{thm:everything}
Let $\eps>0$.
\begin{enumerate}[(a)]
    \item\label{eq:thma} There exists $c_0>0$ such that if $p\in (\eps,1-\eps)$, then w.h.p. $$h(G_1(V)\cup G_3(V))\geq c_0n.$$

    \item\label{eq:thmb} There exists $c_0>0$ such that the following holds. Let $p\in (0,1/2-\eps)$ and $d\geq 3$ be odd such that $q=q_d(p)\geq n^{-0.1}$ and $d\leq c_0\log n$. Then w.h.p.  
    $$h(G_d(V))\geq \frac{c_0^{d}q^{22}D}{n}.$$
\end{enumerate}
\end{theorem}

\begin{proof}
As the proof of the two cases diverge only slightly, we treat them simultaneously. We have the following setup in the two cases. The parameters $d,\alpha,c,K,L,a,b$ are shared across, but may take different values.
\begin{enumerate}[(a)]
    \item Let $\alpha,c,K$ be the values guaranteed by \Cref{lemma:conditions1} with respect to $\eps$, and fix some set $V\subseteq \{0,1\}^n$ satisfying \eqref{a:1}-\eqref{a:6} of \Cref{lemma:conditions1}. Set 
    \[d=3,\quad  L=cn^6, \quad \text{and}\quad a=b=cn^3.\]
    Then we have the following properties of $V$:
    \begin{enumerate}[(i)]
        \item\label{xa:1} $|V|\geq cN$;
        \item\label{xa:2} for every $x\in \{0,1\}^n$, $|S_d(x)\cap V|\geq a$;
        \item\label{xa:3} for every $x\in \{0,1\}^n$, the number of $\alpha$-full elements in $B_d(x)$ is at most $K$;
        \item\label{xa:4} for every $x\in V$, if $x$ is $\alpha$-full, then there are at least $b$ vertices in $S_d(x)$ that are the endpoints of some path of length $3$ in $G_1(V)$ starting at $x$;  
        \item\label{xa:5} for every $x,y\in V$ such that $\dist_H(x,y)=3d$ and $x,y$ are not $\alpha$-full, there are at least $L$ pure paths in $G_d(V)$ between $x$ and $y$;
        \item\label{xa:6} for every $x\in V$, if $x$ is not $\alpha$-full, then $|N_{d,V}(x)|\geq b$.
    \end{enumerate}

    \item  Let $\alpha,c$ be the values guaranteed by \Cref{lemma:conditions2} with respect to $\eps$. We assume that $c_0<\frac{1}{10}(\log\frac{8}{\alpha})^{-1}$, then $d\leq c_0\log n$ implies $(\alpha/8)^d\geq n^{-0.1}$. Therefore, the requirements of \Cref{lemma:conditions2} are satisfied for any odd $d$ allowed by~\eqref{eq:thmb}. Fix $V\subseteq \{0,1\}^n$ satisfying \eqref{b:1}-\eqref{b:5} of \Cref{lemma:conditions2}. Set 
    \[K=0,\quad L=c^dq^7D^2,\quad a=pD/2, \quad\text{and}\quad b=c^dq^7D.\]
    Then we have the following properties of $V$:
    \begin{enumerate}[(i)]
    \item\label{xb:1} $|V|\geq pN/2$;
    \item\label{xb:2} for every $x\in \{0,1\}^n$, $|S_d(x)\cap V|\geq a$;
    \item\label{xb:3} there are no $\alpha$-full vertices.
    \item\label{xb:4} for every $x,y\in V$ with $\dist_H(x,y)=3d$, there exist $L$ pure paths in $G_d(V)$ between $x$ and $y$;
    \item\label{xb:5} for every $x\in V$, $|N_{d,V}(x)|\geq b$.
    \end{enumerate}
\end{enumerate}
In the rest of the proof, we view $V$  as a deterministic object having the appropriate set of properties. As these properties hold w.h.p., it suffices to prove that $h(G_1(V)\cup G_3(V))\geq c_0n$ in case \eqref{eq:thma}, and $h(G_d(V))\geq c_0^{d}q^{22}D$ in case \eqref{eq:thmb} for some $c_0=c_0(\eps)>0$.

\medskip

 First, consider $Q_n^d$. By \Cref{lemma:congQnd}, 
 $$\cng(Q_n^d)\leq \frac{nN}{D}.$$ 
 Therefore, by \Cref{lemma:path-system}, for every $x,y\in \{0,1\}^n$, there is a $Q_n^d$-path $P_{x,y}$ connecting $x$ and $y$ such that every $e\in E(Q_n^d)$ is contained in at most 
 \begin{equation}\label{eq:def:T}
     T:=\frac{2nN}{D}
 \end{equation}
of the paths $P_{x,y}$. For every $x,y\in V$, we redistribute the weight of $P_{x,y}$ using paths in $G_1(V)\cup G_d(V)$.

 Fix some $x,y\in V$ and let $x=x_0,\dots,x_k=y$ be the vertices of $P_{x,y}$, where $k=k_{x,y}$ is the length of the path. For $i=0,\dots,k$, we randomly choose a vertex $z_i\in S_d(x_i)$ from a set~$N_i$ (the non-emptyness of which is proved in~\Cref{claim:Ni} below) in the following manner.

 In case $i=0$ or $i=k$, we proceed depending on whether $x_i$ is $\alpha$-full (which can only happen if we are in case~\eqref{eq:thma}).

\begin{description}
    \item[Case 1.] $x_i$ is $\alpha$-full. 
    
    Let $N_i$ be the set of points $z\in S_d(x_i)$ such that $z$ is not $\alpha$-full, there is a path of length $3$ in $G_1(V)$ between~$x$ and~$z$, and~$x_0z$ avoids~$x_0x_{1}$ when~$i=0$, and~$x_kz$ avoids~$x_{k-1}x_k$ when~$i=k$. Choose~$z_i$ uniformly at random from the set~$N_i$. Let
    \[R_i=R_i^{x,y}\]
    be a path of length~$3$ in~$G_1(V)$ connecting~$x$ and~$z_i$.
    \item[Case 2.] $x_i$ is not $\alpha$-full. 
    
    Let $N_i$ be the set of points $z\in N_{d,V}(x_i)$ such that $z$ is not $\alpha$-full, and~$x_0z$ avoids~$x_0x_{1}$ when~$i=0$ and~$x_kz$ avoids~$x_{k-1}x_k$ when~$i=k$.. Choose~$z_i$ uniformly at random from the set~$N_i$. Let 
    \[R_i=R_i^{x,y}\]
    to be the one edge path between~$x_i$ and~$z_i$.
\end{description}
In case $i=1,\dots,k-1$, we proceed one-by-one as follows. For $i=1,\dots,k-2$, if $z_{i-1}$ is already chosen, we choose~$z_i$ uniformly at random from the set~$N_i$ defined as follows. A vertex~$z\in S_d(x_i)$ is an element of~$N_i$ if~$z$ is not $\alpha$-full, $zx_i$ avoids $x_{i-1}x_i,x_ix_{i+1},x_{i-1}z_{i-1}$, and in case $i=k-1$, also avoids~$x_{k}z_{k}$.

This choice of $z_0,\dots,z_k$ ensures that $\dist_H(z_{i-1},z_{i})=3d$ for $i=1,\dots,k$, and $z_0,\dots,z_k$ are not~$\alpha$-full.
\begin{claim}\label{claim:Ni}
$|N_i|\geq b/2$ for $i=0,\dots,k$.
\end{claim}

\begin{proof}
First, consider the case if $i=0$ or $i=k$. If $x_i$ is $\alpha$-full, then we are in case~\eqref{eq:thma}, and by \eqref{a:4} there is a set $T\subseteq S_d(x_i)$ of size  at least $b=cn^3$ such that every element of $T$ is connected to $x$ by a path of length $3$ in $G_1(V)$.  As $|\supp(x_0\oplus x_{1})|=|\supp(x_{k-1}\oplus x_k)|=3$, there are at most $O(n^{2})$ elements $z\in T$ such that $\supp(z\oplus x_0)\cap \supp(x_0\oplus x_{1})\neq \emptyset$ when~$i=0$ and $\supp(z\oplus x_k)\cap \supp(x_{k-1}\oplus x_{k})\neq \emptyset$ when~$i=k$. Also, there are at most $K$ $\alpha$-full elements in $T$. Thus, $|N_i|\geq |T|-O(n^2)-K\geq b/2$.

If~$x_i$ is not $\alpha$-full, then  $|N_{d,V}(x_i)|\geq b\geq n^{d-0.9}$. Also, at most~$K$ elements of~$N_{d,V}(x_i)$ are $\eps$-full, and there are at most $O(n^{d-1})$ elements $z\in T$ such that $\supp(z\oplus x_0)\cap \supp(x_0\oplus x_{1})\neq \emptyset$ when $i=0$ and $\supp(z\oplus x_k)\cap \supp(x_{k-1}\oplus x_{k})\neq \emptyset$ when $i=k$. Therefore, $|N_i|\geq |N_{d,V}(x_i)|-O(n^{d-1})-K\geq b/2$.

Next, consider the case $i=1,\dots,k-1$. We have $|S_d(x_i)\cap V|\geq a\geq b$. Among the elements $S_d(x_i)\cap V$, at most~$K$ are  $\eps$-full, and there are at most $O(n^{d-1})$ elements $z\in S_d(x_i)$ such that $\supp(z\oplus x_i)$ intersects any of $\supp(x_{i-1}\oplus x_i),\supp(x_i\oplus x_{i+1}),\supp(x_{i-1}\oplus z_{i-1})$, and in case $i=k-1$, also $\supp(x_{k}\oplus z_{k})$. Thus, $|N_i|\geq |S_3(x_i)\cap V|-O(n^{d-1})-K\geq cD/2\geq b/2$.
\end{proof}
The previous claim ensures that for every fixed $z\in S_d(x_i)$, we have 
\begin{equation}\label{eq:Pzzi}
   \mathbb{P}(z=z_i)\leq \frac{2}{b}. 
\end{equation}
Now for $i=1,\dots,k$, let $Z_i=Z_{i}^{x,y}$ be chosen uniformly at random from the set of all the pure paths in~$G_d(P)$ between~$z_{i-1}$ and~$z_{i}$. Let $Z=Z^{x,y}$ be the path that is the union of the paths $R_0,Z_{1},\dots,Z_k,R_k$. Formally, $Z$ is a walk, not necessarily a path, but we can always just take a subset forming a path.

 We define the A-flow $\varphi$ on $G_1(V)\cup G_d(V)$ such that for every path $R$ between $x$ and $y$, we set $$\varphi(R):=\mathbb{P}(Z=R).$$
 In what follows, we study the congestion of $\varphi$. For a fixed edge $e\in E(G_1(V))\cup E(G_d(V))$, we bound the congestion 
 $$\cng_{\varphi}(e)=\sum_{R\in\mathcal{P}:e\in R}\varphi(R)=\sum_{x,y\in V}\mathbb{P}(e\in Z^{x,y}),$$
which is also the expected number of pairs $(x,y)$ for which $e\in Z^{x,y}$.

First, consider the case $e\in E(G_1(V))$. Such an edge is only contained in a path in case~\eqref{eq:thma}. The edge~$e$ is contained in $Z^{x,y}$ if $e\in R_0^{x,y}$ or $e\in R_{k}^{x,y}$.  If $e\in R_0^{x,y}$, 
then $e$ connects $S_j(x)$ and $S_{j+1}(x)$ for some $j\in \{0,1,2\}$.
By~\eqref{eq:Pzzi}, for every path~$R$ of length~$3$ starting with~$x$, we have 
$$\mathbb{P}(R=R_0^{x,y})\leq \frac{2}{b}=\frac{2}{cn^3},$$ and~$e$ can be contained in at most~$O(n^{2-j})$ such paths for a given~$x$. Therefore,
$$\mathbb{P}(e\in R_0^{x,y})=O\Big(\frac{1}{cn^{j+1}}\Big).$$ 
The similar argument holds for the case that $e\in R_k^{x,y}$. 
On the other hand, the number of elements~$x$ such that~$e$ connects~$S_j(x)$ and~$S_{j+1}(x)$ is~$O(n^{j})$.
Thus, we get
\begin{equation}\label{eq:congphie}
    \cng_{\varphi}(e)=\sum_{x,y\in V}\mathbb{P}(e\in Z^{x,y})=\sum_{x,y\in V}\mathbb{P}(e\in R_0^{x,y}\cup R_k^{x,y})=N\cdot\sum_{j=0}^2 O\left(\frac{1}{cn^{j+1}}\cdot n^{j}\right)=O\left(\frac{N}{cn}\right).
\end{equation}

Next, we consider the case $e\in G_d(V)$. By~\Cref{claim:Ni} the probability that $e$ is the first or last edge of~$Z^{x,y}$ is at most $2/b$, and there are at most $|V|\leq N$ values of $x$ or $y$ for which it can be a first or last edge. Therefore,
$$\sum_{x,y\in V}\mathbb{P}(e\mbox{ is a first or last edge of }Z^{x,y})\leq \frac{4\cdot 2N}{b}.$$
Now consider the events that $e$ is an inner edge of a path $Z^{x,y}$. Then $e$ is in some pure path $Z_i^{x,y}$, so 
\begin{align*}
\sum_{x,y\in V}\mathbb{P}(x,y\mbox{ is an inner edge of }Z^{x,y})&\leq \sum_{x,y\in V}\sum_{i=1}^{k_{x,y}}\mathbb{P}(e\in Z^{x,y}_i)\\
&=\sum_{\substack{F:e\in F,\\F\text { is a pure path}}}\sum_{x,y\in V}\sum_{i=1}^{k_{x,y}}\mathbb{P}(F=Z^{x,y}_i).
\end{align*}
For fixed $F$ and $x,y,i$, we have 
$$\mathbb{P}(F=Z^{x,y}_i)\leq\frac{4}{b^2L}.$$
Indeed, if $\tilde{x}$ and $\tilde{y}$ are the endpoints of $F$, then by~\eqref{eq:Pzzi}, $\max(\mathbb{P}(\tilde{x}=z_{i-1}),\mathbb{P}(\tilde{y}=z_{i}))\leq \frac{2}{b}$, and if $\tilde{x}=z_{i-1}$ and $\tilde{y}=z_i$, then the probability that $F$ is chosen among the at least $L$ pure paths between $\tilde{x}$ and $\tilde{y}$ is at most $1/L$. Moreover, $\mathbb{P}(F=Z^{x,y}_i)=0$, unless $\dist_H(\tilde{x},\tilde{y})=3d$ and the edge~$x_{i-1}x_i$ of $P_{x,y}$ satisfies $\dist_H(x_{i-1},\tilde{x})=\dist_H(x_i,\tilde{y})=d$. When $\dist_H(\tilde{x},\tilde{y})=3d$, the number of edges $st\in E(Q_n^d)$ such that $\dist_H(s,\tilde{x})=d$ and $\dist_H(t,\tilde{y})=d$ is at most $\binom{3d}{d,d,d}= \frac{(3d)!}{(d!)^3}\leq 3^{3d}$. Moreover, by~\eqref{eq:def:T},  each such edge~$st$ is contained in at most $T$ paths $P_{x,y}$, showing that
$$\sum_{x,y\in V}\sum_{i=1}^{k_{x,y}}\mathbb{P}(F=Z^{x,y}_i)\leq 3^{3d} T\cdot \frac{4}{b^2L}.$$
Finally, $e$ is contained in at most $O(3^{3d}D^4)$ pure paths by \Cref{lemma:path_count}, so
$$\sum_{\substack{F:e\in F,\\F\text { is a pure path}}}\sum_{x,y\in V}\sum_{i=1}^{k_{x,y}}\mathbb{P}(F=Z^{x,y}_i)\leq O\Big(3^{3d} T\cdot \frac{4}{b^2L}\cdot 3^{3d}D^4\Big)=O\left(\frac{3^{6d}D^3nN}{b^2L}\right).$$
 Therefore, we get the following bound for the congestion of $e$:
\begin{align}
\cng_{\varphi}(e)&\leq \sum_{x,y\in V}\mathbb{P}(e\mbox{ is a first or last edge of }Z^{x,y})+\mathbb{P}(x,y\mbox{ is an inner edge of }Z^{x,y}) \nonumber\\
&\leq  O\left(\frac{N}{b}+\frac{3^{6d}D^3nN}{b^2L}\right)=\begin{cases}O\left(N/(c^3n^2)\right)  &\mbox{in case~\eqref{eq:thma}},\\
O\left(3^{6d}nN/(c^{3d}q^{21}D)\right) &\mbox{in case~\eqref{eq:thmb}}.
\end{cases}\label{eq:twocases}
\end{align}
In case ~\eqref{eq:thma}, summarizing~\eqref{eq:congphie} and~\eqref{eq:twocases}, we showed that if $e\in G_1(V)$, then $\cng_{\varphi}(e)=O(N/(cn))$, and if  $e\in G_3(P)$, then  $\cng_{\varphi}(e)=O(N/(c^3n^2))$. Therefore, $\cng(\varphi)=O(N/(cn))$. Thus, by \Cref{thm:congestion}, 
$$h(G_1(V)\cup G_3(V))\geq \frac{|V|}{2\cng(G_1(V)\cup G_d(V))}\geq \frac{cN}{O(N/cn)}=\Omega(c^2 n).$$ 
In case~\eqref{eq:thmb}, we have $\cng(\varphi)=O(3^{6d}nN/(c^{3d}q^{21}D)$. Thus, by \Cref{thm:congestion}, 
$$h(G_d(V))\geq \frac{|V|}{2\cng(G_d(V))}\geq \frac{pN}{O(3^{6d}nN/(c^{3d}q^{21}D)}\geq \Omega\left(\left(\frac{c^{3}}{3^6}\right)^d\cdot \frac{q^{22}D}{n}\right).$$
This completes the proof of~\Cref{thm:everything}.
\end{proof}

\begin{proof}[Proof of \Cref{thm:main}] 
\eqref{eq:maina} $G_1(V)\cup G_3(V)$ is a subgraph of $G$, so we have $h(G)\geq h(G_1(V)\cup G_3(V))$. By \Cref{thm:everything}, w.h.p. $h(G_1(V)\cup G_3(V))\geq cn$ for some $c>0$ only depending on $\eps$.

\eqref{eq:mainb}  By \Cref{thm:everything}, there exists $c_0\in (0,1)$ such that for every $p\in (0,1/2-\eps)$ and odd $d$ satisfying $$d\leq c_0\log n\quad\mbox{ and }\quad q=p(1-p)^{2^d-2}\geq n^{-0.1},$$ we have w.h.p. $$h(G_d(V))\geq \frac{c_0^{d}q^{22}D}{n}.$$ 
Let $p\in (n^{-0.05},1/2-\eps)$, and choose
$$d=\left\lfloor \frac{1}{2}\log_2\log n+c_0\log_2 (1/p)\right\rfloor.$$
If $d$ is even, replace $d$ with $d-1$. Then we have $d\leq c_0\log n$ and
$$p(1-p)^{2^d-2}\geq n^{-0.05}(1-p)^{2^d}\geq n^{-0.05}(1-p)^{\frac{1}{p}(\log n)^{1/2}}=n^{-0.05+o(1)}.$$
Thus, we have w.h.p.
$$h(G_d(V))\geq \frac{c_0^{d}q^{22}D}{n}\geq c_0^{c_0\log n} n^{-2.2-1}\left(\frac{n}{d}\right)^d>n^{d/2}.$$
Therefore, w.h.p.
$$h(G)\geq h(G_d(V))\geq n^{d/2}\geq n^{\Omega(\log\log n +\log (1/p))},$$
which completes the proof.
\end{proof}

%\section*{Acknowledgments}

%IT acknowledges the support of the Swedish Research Council grant VR 2023-03375.
\end{document}